\documentclass[12pt]{article}

\title{Computational Multiscale Methods for Linear Heterogeneous Poroelasticity\thanks{The authors acknowledge support from the  Germany/Hong Kong Joint Research Scheme sponsored by the German Academic Exchange Service (DAAD) under the project 57334719 and the Research Grants Council of Hong Kong with reference number G-CUHK405/16. Further, the authors thank the Hausdorff Institute for Mathematics in Bonn for the kind hospitality during the trimester program on multiscale problems in 2017.}}
\author{Robert Altmann\thanks{Department of Mathematics, University of Augsburg.}\ , Eric Chung\thanks{Department of Mathematics, The Chinese University of Hong Kong.}\ , Roland Maier$^\dagger$,\\ Daniel Peterseim$^\dagger$\thanks{supported by the Sino-German Science Center on the occasion of the Chinese-German Workshop on Computational and Applied Mathematics in Shanghai 2017.}, Sai-Mang Pun$^\ddagger$}

\usepackage{amscd,amssymb}
\usepackage{amsmath}
\usepackage{graphicx}
\usepackage{subfigure,fancybox}
\usepackage{epic,eepic}
\usepackage{amstext}
\usepackage{pstricks}
\usepackage{pst-node,pst-plot,pst-coil}
\usepackage{amsthm}
\usepackage{amsmath}
\usepackage[utf8]{inputenc}
\usepackage[english]{babel}
\usepackage[square,numbers]{natbib}
\usepackage{booktabs}
\usepackage{comment}
\usepackage{mathtools}
\usepackage{tikz,pgfplots}

\DeclareFontEncoding{FMS}{}{}
\DeclareFontSubstitution{FMS}{futm}{m}{n}
\DeclareFontEncoding{FMX}{}{}
\DeclareFontSubstitution{FMX}{futm}{m}{n}
\DeclareSymbolFont{fouriersymbols}{FMS}{futm}{m}{n}
\DeclareSymbolFont{fourierlargesymbols}{FMX}{futm}{m}{n}
\DeclareMathDelimiter{\VERT}{\mathord}{fouriersymbols}{152}{fourierlargesymbols}{147}

\newcommand\norm[1]{\Vert#1\Vert}

\newtheorem{theorem}{Theorem}[section]

\newtheorem{lemma}[theorem]{Lemma}

\theoremstyle{definition}

\theoremstyle{remark}

\definecolor{myBlue}{RGB}{30,144,255} 
\definecolor{myGreen}{RGB}{69,169,0} 
\definecolor{myRed}{RGB}{165,12,42} 
\definecolor{myOrange}{RGB}{225,92,22} 

\newcommand{\delete}[1]{ }

\newcommand\calA{\mathcal A}
\newcommand\calB{\mathcal B}

\newcommand\bilA{\mathfrak a}
\newcommand\calCf{\mathcal C_\text{fs}}

\newcommand\sumjn{\sum_{j=1}^n}

\newcommand\R{\mathbb R}

\newcommand\V{V}   
\newcommand\Q{Q}   

\newcommand\Vh{V_h}  
\newcommand\Qh{Q_h}  
\newcommand\VH{V_H}  
\newcommand\QH{Q_H}  

\newcommand\Vms{V_\text{ms}}
\newcommand\Qms{Q_\text{ms}}

\newcommand\Vf{V_\text{fs}}   
\newcommand\Qf{Q_\text{fs}}   

\newcommand\Ru{R_\text{ms}^1}
\newcommand\Rp{R_\text{ms}^2}
\newcommand\RS{\tilde R_\text{ms}}     
\newcommand\RuS{\tilde R_\text{ms}^1}  
\newcommand\RpS{\tilde R_\text{ms}^2}  
\newcommand\RfS{\tilde R_\text{fs}}     

\newcommand\uh{u_h^n}           
\newcommand\uProj{\Ru u_h^n}    
\newcommand\ums{u_\text{ms}^n}  
\newcommand\umsS{\tilde u_\text{ms}^n}  
\newcommand\ufS{u_\text{fs}^n}   
\newcommand\ph{p_h^n}           
\newcommand\pProj{\Rp p_h^n}    
\newcommand\pms{p_\text{ms}^n}  
\newcommand\pmsS{\tilde p_\text{ms}^n}  

\newcommand\uhj{u_h^j}          
\newcommand\phj{p_h^j}          

\newcommand\errRu{\rho_u^n}
\newcommand\errRp{\rho_p^n}
\newcommand\errEu{\eta_u^n}
\newcommand\errEp{\eta_p^n}

\newcommand\errEupre{\eta_u^{n-1}}  
\newcommand\errEppre{\eta_p^{n-1}}
\newcommand\errRuj{\rho_u^j}        
\newcommand\errRpj{\rho_p^j}
\newcommand\errEuj{\eta_u^j}	    
\newcommand\errEpj{\eta_p^j}

\newcommand\errEujj{\eta_u^{j-1}}	    
\newcommand\errEpjj{\eta_p^{j-1}}
\newcommand\errEuinit{\eta_u^0}     
\newcommand\errEpinit{\eta_p^0}

\newcommand\D{D_\tau}

\newcommand\datan{\ddd^n}
\DeclareMathOperator{\ddd}{data}

\newcommand\dt{\,\text{d}t}
\newcommand\dx{\,\text{d}x}

\begin{document}

\maketitle
\begin{abstract}
We consider a strongly heterogeneous medium saturated by an incompressible viscous fluid as it appears in geomechanical modeling. This poroelasticity problem suffers from rapidly oscillating material parameters, which calls for a thorough numerical treatment. In  this  paper, we propose a method based on the local orthogonal decomposition technique and motivated by a similar approach used for linear thermoelasticity. Therein, local corrector problems are constructed in line with the static equations, whereas we propose to consider the full system. This allows to benefit from the given saddle point structure and results in two decoupled corrector problems for the displacement and the pressure. We prove the optimal first-order convergence of this method and verify the result by numerical experiments.
\end{abstract}
%
%
\section{Introduction}
Modeling the deformation of porous media saturated by an incompressible viscous fluid is of great importance for many physical applications such as reservoir engineering in the field of geomechanics \cite{Zob10} or the modeling of the human anatomy for medical applications \cite{Mura2016,doi:10.1137/130921866}. To obtain a reasonable model, it is important to couple the flow of the fluid with the behavior of the surrounding solid.
Biot proposed a model that couples a Darcy flow with linear elastic behavior of the porous medium \cite{Biot41}. The corresponding analysis was given in \cite{Sho00}. For this so-called poroelastic behavior, pressure and displacement are averaged across (infinitesimal) cubic elements such that pressure and displacement can be treated as variables on the entire domain of interest. Furthermore, the model is assumed to be quasi-static, i.e., an internal equilibrium is preserved at any time. In the poroelastic setting, this means that volumetric changes occur slowly enough for the pressure to remain basically constant throughout an infinitesimal element.

If the given material is homogeneous, the poroelastic behavior can be simulated using standard numerical methods such as the finite element method, see for instance \cite{ErnM09}. However, if the medium is strongly heterogeneous, the material parameters may oscillate on a fine scale. In such a scenario, the classical finite element method only yields acceptable results if the fine scale is resolved by the spatial discretization, which is unfeasible in practical applications.  
To overcome this issue, homogenization techniques may be applied, such as the general multiscale finite element method (GMsFEM) \cite{EGH13}, used in \cite{BroV16a, BroV16b}, or the localized orthogonal decomposition technique (LOD) \cite{MalP14} as used in \cite{MalP17} for the similar problem of linear thermoelasticity. The general idea of these methods is to construct low-dimensional finite element spaces which incorporate spatial fine scale features using adapted basis functions. This involves additional computations in the offline stage with the benefit of having much smaller linear systems to solve in every time step due to the much lower amount of degrees of freedom.

In the present paper, a multiscale finite element method is proposed based on the LOD method and adopting the ideas presented in \cite{MalP17}. In contrast to the method of \cite{MalP17}, we are able to exploit the saddle point structure of the problem in order to obtain fully symmetric and decoupled corrector problems without the need for additional corrections. 
Furthermore, this implies that the correctors are independent of the Biot-Willis fluid-solid coupling coefficient, although it may vary rapidly as well.  

The work is structured as follows. In Section~\ref{sec:linporo} we present the model problem and introduce the necessary notation. Section~\ref{sec:numapprox} is devoted to the discretization of the problem. This includes the classical finite element method on a fine mesh as well as the formulation introduced in \cite{MalP17} translated to poroelasticity. We then introduce the decoupled corrector problems and the resulting new multiscale scheme for which we prove convergence. Finally, numerical results, which illustrate the theoretical findings, are presented in Section~\ref{sec:numex}.

Throughout the paper $C$ denotes a generic constant, independent of spatial discretization parameters and the time step size. Further, $a \lesssim b$ will be used equivalently to $a \leq C b$. 
%
%
\section{Linear Poroelasticity}\label{sec:linporo}
%
\subsection{Model problem}
We consider the linear poroelasticity problem in a bounded and polyhedral Lipschitz domain $D \subset \mathbb{R}^d$ ($d = 2,3$) as discussed in \cite{Sho00}. For the sake of simplicity, we restrict ourselves to homogeneous Dirichlet boundary conditions. The extension to Neumann boundary conditions is straightforward. 
This means that we seek the pressure ${p\colon[0,T] \times D \rightarrow \mathbb{R}}$ and the displacement field $u\colon [0,T] \times D \rightarrow \mathbb{R}^d$ within a given time $T>0$ such that 
\begin{subequations}
\label{eqn:poro}
\begin{eqnarray}
    -\nabla \cdot \big( \sigma (u) \big) + \nabla (\alpha p) &=& 0 \quad \text{in } (0,T] \times D, \label{eqn:pde_1}\\
    \partial_t\Big(\alpha \nabla \cdot u + \frac{1}{M} p \Big)- \nabla \cdot \Big( \frac{\kappa}{\nu} \nabla p \Big) &=& f \quad \text{in } (0,T] \times D \label{eqn:pde_2}
\end{eqnarray}    
with boundary and initial conditions    
\begin{eqnarray}
    u &=& 0 \quad\ \text{on } (0,T] \times \partial D, \label{eqn:bc_1}\\
    p &=& 0 \quad\ \text{on } (0,T] \times \partial D, \label{eqn:bc_2}\\
    p(\cdot,0) &=& p^0 \quad \text{in } D \label{eqn:ic}.
\end{eqnarray}
\end{subequations}
In the given model, the primary sources of the heterogeneities in the physical properties arise from the stress tensor $\sigma$, the permeability $\kappa$, and the Biot-Willis fluid-solid coupling coefficient $\alpha$. Further, we denote by $M$ the Biot modulus and by $\nu$ the fluid viscosity. The source term $f$ represents an injection or production process. In the case of a linear elastic stress-strain constitutive relation, we have that the stress tensor and symmetric strain gradient may be expressed as
$$ \sigma(u) = 2\mu \varepsilon (u) + \lambda (\nabla \cdot u) \mathcal{I}, \qquad \varepsilon (u) = \frac{1}{2}\big(\nabla u + (\nabla u)^T \big),$$
where $\mu$ and $\lambda$ are the Lam\'e coefficients and $\mathcal{I}$ is the identity tensor. In the case where the media is heterogeneous the coefficients $\mu$, $\lambda$, $\kappa$, and $\alpha$ may be highly oscillatory.
%
\subsection{Function spaces}
In this subsection we clarify the notation used throughout the paper. We write $(\cdot,\cdot)$ to denote the inner product in $L^2(D)$ and $\norm{\cdot}$ for the corresponding norm. Let $H^1(D)$ be the classical Sobolev space with norm $\norm{v}_1^2 := \norm{v}_{H^1(D)}^2 = \norm{v}^2 + \norm{\nabla v}^2$ and let $H_0^1(D)$ be the subspace with functions having a vanishing trace. The corresponding dual space is denoted by $H^{-1}(D)$. Moreover, we write $L^p(0,T;X)$ for the Bochner space with the norm
\begin{align*}
  \norm{v}_{L^p(0,T;X)} = \bigg(\int_0^T \norm{v &}_X^p \dt\bigg)^{1/p}, \quad 1\leq p < \infty, \\
  \norm{v}_{L^\infty(0,T;X)} &= \sup_{0\leq t\leq T} \norm{v}_X,
\end{align*}
where $X$ is a Banach space equipped with the norm $\norm{\cdot}_X$. The notation $v\in H^1(0,T; X)$ is used to denote that both $v$ and $\partial_t v$ are elements of the space $L^2(0,T;X)$. To shorten notation we define the spaces for the displacement and the pressure by 
$$ \V := \big[ H^1_0(D) \big]^d, \qquad \Q := H^1_0(D).$$
%
\subsection{Variational formulation}
In this subsection we give the corresponding variational formulation of the poroelasticity system \eqref{eqn:poro}. To obtain a variational form we multiply the equations \eqref{eqn:pde_1} and \eqref{eqn:pde_2} with test functions from $\V$ and $\Q$, respectively, and use Green's formula together with the boundary conditions \eqref{eqn:bc_1} and \eqref{eqn:bc_2}. This leads to the following problem: find $u(\cdot,t)\in \V$ and $p(\cdot,t)\in \Q$ such that
\begin{subequations}
\label{eqn:var}
\begin{eqnarray}
    a(u,v) - d(v, p) &=& 0, \label{eqn:var_f_1} \\
    d(\partial_t u, q) + c(\partial_t p,q) +b(p,q) &=& (f,q), \label{eqn:var_f_2} 
\end{eqnarray}
for all $v\in \V$, $q \in \Q$ and  
\begin{eqnarray}
    p(\cdot,0) = p^0. \label{eqn:var_ic}
\end{eqnarray}
\end{subequations}
The bilinear forms $a\colon \V\times\V\to\R$, $b,c\colon \Q\times\Q\to\R$, and $d\colon \V\times\Q\to\R$ are defined through 
$$ a(u,v) := \int_D \sigma(u) : \varepsilon(v) \dx, \qquad b(p,q) := \int_D \frac{\kappa}{\nu}\, \nabla p \cdot \nabla q \dx,$$
$$ c(p,q) := \int_D \frac{1}{M}\, pq \dx, \qquad d(u,q) := \int_D \alpha\, (\nabla \cdot u)q \dx.$$
Note that \eqref{eqn:var_f_1} can be used to define a consistent initial value $u^0:=u(\cdot, 0)$. Using Korn's inequality \cite{Cia88}, we have the bounds 
\begin{equation}\label{eq:bounds_a}
 c_{\sigma} \norm{v}_{1}^2 \leq a(v,v) \leq C_{\sigma} \norm{v}_{1}^2 
\end{equation}
for all $v\in \V$, where $c_\sigma$ and $C_\sigma$ are positive constants. Similarly, there are positive constants $c_\kappa$ and $C_\kappa$ such that 
\begin{equation}\label{eq:bounds_b}
c_{\kappa} \norm{q}_{1}^2 \leq b(q,q) \leq C_{\kappa} \norm{q}_{1}^2
\end{equation}
for all $q \in \Q$. 
We write $\norm{\cdot}_a$ for the energy norm induced by the bilinear form $a(\cdot,\cdot)$ and similarly $\norm{\cdot}_b$ for the norm induced by $b$. 
The existence and uniqueness of solutions $u$ and $p$ to \eqref{eqn:var} have been discussed and proved in \cite{Sho00}. 
%
%
\section{Numerical Approximation}\label{sec:numapprox}
In this section, we present different finite element schemes for the discretization of system \eqref{eqn:var}. The classical method is only meaningful if oscillations are resolved by the underlying mesh and this approach will solely serve as a reference.
The main goal of this section is to approximate the solution on a mesh of some feasible coarse scale of resolution independent of microscopic oscillations. Our construction is based on the concept of {\em localized orthogonal decomposition} (LOD)  \cite{MalP14,Brown.Peterseim:2014,Pet16,doi:10.1137/16M1088533} and adapts ideas from thermoelasticity \cite{MalP17} and the heat equation \cite{MalP18}.
%
\subsection{Fine-scale discretization using classical FEM}
We define appropriate finite element spaces for the poroelasticity system. Let $\{\mathcal{T}_h\}_{h>0}$ be a shape regular family of meshes \cite{Cia78} for the computational domain $D$ with mesh size $h_K:=\text{diam}(K)$ for $K \in \mathcal{T}_h$. We define the maximal diameter by $h := \max_{K\in \mathcal{T}_h}h_K$. Based on the mesh, we define the piecewise affine finite element spaces
$$ \Vh := \{ v\in \V : v|_K \text{ is a polynomial of degree} \leq 1 \text{ for all } K \in \mathcal{T}_h \},$$
$$ \Qh := \{ q\in\Q : q|_K \text{ is a polynomial of degree} \leq 1 \text{ for all } K \in \mathcal{T}_h \}.$$
For the discretization in time, we consider a uniform time step $\tau$ such that $t_n = n\tau$ for $n\in \{0,1,\cdots,N\}$ and $T = N\tau$.

Using the notation introduced above, we discretize system \eqref{eqn:var} with a backward Euler scheme in time and by finite elements in space, i.e., for $n\in\{1,\cdots,N\}$ we aim to find $u_h^n\in \Vh$ and $p_h^n \in \Qh$ such that 
\begin{subequations}
\label{eqn:var_h}
\begin{eqnarray}
    a(u_h^n, v) -d(v, p_h^n) &=& 0, \label{eqn:var_d_1} \\
    d(\D u_h^n,q) + c(\D p_h^n,q) + b(p_h^n,q) &=& (f^n,q) \label{eqn:var_d_2}
\end{eqnarray}
\end{subequations}
for all $v \in \Vh$ and $q \in \Qh$. Within the equations, $\D$ denotes the discrete time derivative, i.e., $\D u_h^n := (u_h^n-u_h^{n-1})/\tau$, and $f^n := f(t_n)$. As initial value we choose $p_h^0\in \Qh$ to be a suitable approximation of $p^0$. Similarly as before, $u_h^0$ is uniquely determined by the variational problem 
$$ a(u_h^0,v) = d (v, p_h^0)$$
for all $v\in \Vh$. 
\begin{lemma}\label{lemma:wp1} 
Given initial data $u_h^0\in \Vh$ and $p_h^0\in \Qh$, system \eqref{eqn:var_h} is well-posed, i.e., there exists a unique solution, which is bounded in terms of the initial values and the source term $f$.
\end{lemma}
\begin{proof}
Observe that for the bilinear form $a$ it holds that 
\begin{align}
  2\,a(u_h^n, u_h^n-u_h^{n-1}) 
  &= a(u_h^n,u_h^n) + a(u_h^n-u_h^{n-1},u_h^n-u_h^{n-1}) - a(u_h^{n-1},u_h^{n-1}) \notag\\
  &\ge \norm{u_h^n}^2_a - \norm{u_h^{n-1}}^2_a.  \label{eq:bilsymform}
\end{align}
A similar result can be shown for the bilinear form $c$. Choosing $v = u_h^n-u_h^{n-1} \in \Vh$ as test function in \eqref{eqn:var_d_1} and $q = \tau p_h^n \in \Qh$ in \eqref{eqn:var_d_2} and adding both equations, we obtain
\begin{equation}\label{eq:add1+2}
 a(u_h^n,u_h^n-u_h^{n-1}) + c(p_h^n-p_h^{n-1},p_h^{n}) + \tau\, b(p_h^n,p_h^n) = \tau\, (f^n, p^n_h). 
\end{equation}
Inequality \eqref{eq:bilsymform}, an application of {Young's inequality}, and \eqref{eq:bounds_b} then imply 
\begin{equation*}
 \|u_h^n\|_a^2 + \|p_h^n\|_c^2 + \tau\|p_h^n\|_b^2  \le \frac{\tau}{ c_{\kappa}}\|f^n\|^2 + \|u_h^{n-1}\|_a^2 + \|p_h^{n-1}\|_c^2.
\end{equation*}
A summation over all $n$ finally leads to the stability estimate
\begin{equation*}
  \|u_h^n\|_a^2 + \|p_h^n\|_c^2 + \tau\sum_{j=1}^n\|p_h^j\|_b^2 
  \le \frac{\tau}{ c_{\kappa}} \sum_{j=1}^n\|f^j\|^2 + \|u_h^{0}\|_a^2 + \|p_h^{0}\|_c^2.
\end{equation*}
This implies the uniqueness of the solutions $u_h^n$ and $p_h^n$. Existence follows from the fact that system \eqref{eqn:var_h} is equivalent to a square system of linear equations and, hence, uniqueness implies existence.
\end{proof}
For the presented fine scale discretization in \eqref{eqn:var_h}, one can show the following stability result, which will be important for the convergence proof of the LOD method in Section~\ref{ss:alternLOD}.
\begin{theorem}[{\cite[Th.~3.3]{MalP17}}]
\label{thm:stability}
Assume $f \in L^\infty(0, T;L^2(D)) \cap H^1(0, T;H^{-1}(D))$. Then, the fully discrete solution $(\uh, \ph)$ of \eqref{eqn:var_h} satisfies for all $n=1,\dots, N$ the stability bound 
\[
  \Big( \tau\sumjn \Vert \D \uhj \Vert_1^2 \Big)^{1/2}
  + \Big( \tau\sumjn \Vert \D \phj \Vert^2 \Big)^{1/2} + \Vert \ph \Vert_1
  \lesssim \Vert p_h^0\Vert_1 + \Vert f\Vert_{L^2(0, t_n;L^2(D))}.
\]
Further, in the case $p_h^0 = 0$ we have 
\[
  \Vert\D \uh\Vert_1 + \Vert \D \ph \Vert 
  + \Big( \tau\sumjn \Vert \D \phj \Vert_1^2 \Big)^{1/2}
  \lesssim 
   \Vert f\Vert_{L^\infty(0, t_n;L^2(D))} + \Vert \dot f\Vert_{L^2(0, t_n;H^{-1}(D))}      
\]	
and for $f=0$ it holds that	
\[
  \Vert\D \uh\Vert_1 + \Vert \D \ph \Vert + t_n^{1/2} \Vert \D \ph \Vert_1
  \lesssim t_n^{-1/2} \Vert p_h^0 \Vert_1.   
\]
\end{theorem}
The following theorem states the expected convergence order of $h+\tau$. However, the involved constant for the spatial discretization scales with $\epsilon^{-1}$, which makes this approach unfeasible in  oscillatory media with period $\epsilon$. 
\begin{theorem}[cf.~{\cite[Th.~3.1]{ErnM09}}]
\label{thm:finescale}
Assume that the coefficients satisfy $\mu, \lambda, \kappa, \alpha \in W^{1,\infty}(D)$. Further, let the exact solution $(u,p)$ of \eqref{eqn:poro} be sufficiently smooth and  $(\uh, \ph)$ the fully discrete solution obtained by \eqref{eqn:var_h} for $n=1, \dots, N$. Then, the error is bounded by 
\[
  \Vert u(t_n) - \uh \Vert_1 + \Vert p(t_n) - \ph \Vert 
  + \Big( \tau \sumjn \Vert p(t_j) - p_h^j\Vert_1^2 \Big)^{1/2}
  \le C_\epsilon h + C \tau,
\]
where the constant $C_\epsilon$ scales with $\max \{ \Vert \mu\Vert_{W^{1,\infty}(D)}, \Vert \lambda\Vert_{W^{1,\infty}(D)}, \Vert \kappa\Vert_{W^{1,\infty}(D)}, \Vert \alpha \Vert_{W^{1,\infty}(D)} \}$. 
\end{theorem}
%
\subsection{A multiscale method for poroelasticity}\label{ss:schweden}
Within this subsection, we derive the framework of a generalized finite element method for poroelasticity. First, we introduce $\VH$ and $\QH$ analogously to $\Vh$ and $\Qh$ with a larger mesh size $H>h$. Second, we assume that the family  $\{\mathcal{T}_H\}_{H>0}$ is quasi-uniform and that $\mathcal{T}_h$ is a refinement of $\mathcal{T}_H$ such that $\VH \subseteq \Vh$ and $\QH \subseteq \Qh$. The goal is to construct a new function space with the same dimension as $\VH \times \QH$ but with better approximation properties. For this, we follow the methodology of LOD \cite{MalP14, Pet16} and, in particular, translate the results from thermoelasticity presented in \cite{MalP17} to the present setting. 
%
\subsubsection{Multiscale spaces and projections}\label{sss:schweden:spaces}
Consider a quasi-interpolation $I_H \colon \Vh \rightarrow \VH$ with $I^2_H = I_H$, which satisfies the stability estimate that for $v \in \Vh$ we have 
$$
  H_K^{-1} \norm{v-I_H v}_{L^2(K)} + \norm{\nabla I_H v}_{L^2(K)} 
  \lesssim \norm{\nabla v}_{L^2(\omega_K)}
$$
for all $K \in \mathcal{T}_H$ of size $H_K$ and the element patch 
$$
  \omega_K := \text{int} \bigcup_{\hat{K} \in \mathcal{T}_H,\, \hat{K} \cap K \neq \emptyset} \hat K.$$ 
Similarly, we have a corresponding quasi-interpolation operator on $\Qh$, which we also denote by $I_H$. 
Since the mesh is assumed to be shape regular, the stability estimate above is also global, i.e.,
\begin{equation}\label{eq:corstab}
 H^{-1} \norm{v-I_H v} + \norm{\nabla I_H v} \lesssim \norm{\nabla v},
\end{equation}
where the involved constant depends only on the shape regularity of the mesh. 
%
One typical example of such a quasi-interpolation that satisfies the above assumptions is $I_H = E_H \circ \Pi_H$. Here, $\Pi_H$ denotes the piecewise $L^2$-projection onto $P_1(\mathcal{T}_H)$ (or $P_1(\mathcal{T}_H)^d$ respectively), the space of functions that are affine on each triangle $K \in \mathcal{T}_H$. Moreover, $E_H$ is an averaging operator mapping $P_1(\mathcal{T}_H)$ into $\QH$ (or $P_1(\mathcal{T}_H)^d$ into $\VH$) by
$$ 
  \big(E_H(v)\big)(z) := \frac{1}{\text{card}\{K \in \mathcal{T}_H: z\in K\}} \sum_{K\in \mathcal{T}_H,\, z\in K} v|_K(z), 
$$
where $z$ is a free node in $\QH$. For further details and other available options for $I_H$ we refer to \cite{Pet16}.

Next, we define the kernel of $I_H$ in $\Vh$ and $\Qh$, respectively, by 
$$
  \Vf:= \{ v\in \Vh: I_H v = 0\}, \quad \Qf:= \{ q\in \Qh: I_H q = 0\}.
$$
These kernels are fine scale spaces in the sense that they contain all features that are not captured by the coarse spaces $\VH$ and $\QH$. 
The interpolation operator $I_H$ leads to the decompositions $\Vh = \VH \oplus \Vf$ and $\Qh = \QH \oplus \Qf$, meaning that any function $v \in \Vh$ can be uniquely decomposed into $v = v_{H} + v_\text{fs}$ with $v_{H} \in \VH$ and $v_\text{fs} \in \Vf$ and similarly for $q\in\Qh$.
In the following, we consider two projections on the fine scale spaces, the so-called correctors $\calCf^1\colon \Vh\to\Vf$ and $\calCf^2\colon \Qh\to\Qf$ based on the elliptic bilinear forms $a$ and $b$. They are defined by  
\begin{align*}
  a(\calCf^1 u, v) = a(u, v),  \qquad
  b(\calCf^2 p, q) = b(p, q)
\end{align*}
for all $v\in \Vf$ and $q\in \Qf$. 
With these correctors we define the new finite element spaces
$$
  \Vms:= \{ v_H - \calCf^1 v_H : v_H\in\VH \}, \quad \Qms := \{ q_H - \calCf^2 q_H : q_H\in\QH\},
$$
which have the same dimensions as $\VH$ and $\QH$, respectively. Note that this gives the $a$-orthogonal decomposition $\Vh = \Vms \oplus \Vf$ as well as the $b$-orthogonal decomposition $\Qh = \Qms \oplus \Qf$. 
%
\subsubsection{Multiscale method and convergence}
With all tools in hand, we are now able to present the method proposed in 
\cite{MalP17} for thermoelasticity. Since the considered system involves time derivatives, the corrector problem needs to be solved in each time step. This is, however, too expensive for a computational approach. Therefore, we restrict the computations of the multiscale correctors to the stationary system. This then leads to the definition of $\RS = (\RuS, \RpS)\colon \Vh \times \Qh \rightarrow \Vms \times \Qms$, which we use to construct the space of trial functions. For given $u \in \Vh$ and $p\in \Qh$ we define $\RuS(u, p) \in \Vms$ and $\RpS(p) \in\Qms$ by
\begin{subequations}
\begin{eqnarray}
  a(\RuS(u,p), v) - d(v, \RpS p) &=& a(u, v)-d(v,p), \label{eqn:corr_1} \\
  b(\RpS p, q) &=& b(p, q) \label{eqn:corr_2}
\end{eqnarray}
\end{subequations}
for all $v \in \Vms$ and $q\in \Qms$. Note that the projections $\RuS$ and $\RpS$ are coupled and that we have the relation $\RpS = 1-\calCf^2$. Since $\RuS$ depends on the bilinear form $d$, we need a second projection $\RfS\colon \Qh \to \Vf$ defined by  
\begin{eqnarray}
  a(\RfS\, p, v) = - d(v, \RpS p)  \label{eqn:modify}
\end{eqnarray}
for all $v\in \Vf$. We emphasize that this is a fine scale correction. 

The resulting multiscale discretization (combined with a backward Euler scheme in time) then has the following form. For all $n = 1, \dots, N$ find $\umsS = \ums + \ufS$ with $\ums \in \Vms$, $\ufS \in \Vf$ and $\pmsS \in \Qms$ such that
\begin{subequations}
\label{eqn:gfemS}	
\begin{eqnarray}
  a(\umsS, v) - d(v, \pmsS) &=& 0, \label{eqn:gfem_time_1}\\
  d(\D \umsS, q) + c(\D \pmsS, q) + b(\pmsS, q) &=& (f^n,q), \label{eqn:gfem_time_2}\\
  a(\ufS, w) + d( w, \pmsS) &=& 0  \label{eqn:gfem_time_3}
\end{eqnarray}
\end{subequations}
for all test functions $v \in \Vms$, $q\in \Qms$, and $w\in \Vf$. The initial value is given by $\tilde p_\text{ms}^0 = \RpS p_h^0$. Moreover, we define $\tilde{u}_\text{ms}^0 = u_\text{ms}^0 + u_\text{fs}^0$ with $u_\text{fs}^0 \in \Vf$ given by \eqref{eqn:gfem_time_3} and $u_\text{ms}^0 \in \Vms$ defined through the variational problem 
\[
  a(u_\text{ms}^0, v)
  = a(\tilde u_\text{ms}^0, v) 
  = d(v, \tilde p_\text{ms}^0) 
\]
with test functions $v\in \Vms$. System~\eqref{eqn:gfemS} is well-posed and the errors $\Vert \uh - \umsS\Vert_1$ and $\Vert \ph - \pmsS\Vert_1$ are bounded by an $\epsilon$-independent constant times $H$, see \cite[Th.~5.2]{MalP17}. Together with Theorem~\ref{thm:finescale} this implies that the multiscale solution $(\umsS, \pmsS)$ approximates the exact solution $(u,p)$ with order $H$ in space and order $\tau$ in time. 
Moreover, one may manipulate system~\eqref{eqn:gfemS} in such a way that, in practice, one does not need to compute a fine scale correction in each step. 

%
\subsection{An alternative multiscale method}\label{ss:alternLOD}
We have seen that the method proposed in the previous subsection needs an additional correction on the fine mesh. As explained in \cite{MalP17}, the additional fine scale correction only needs to be computed in the offline stage by using a set of basis functions. This keeps the coarse structure of the system in each time step at the expense of slightly more complicated systems. We now propose a simpler method, which exploits the symmetry of the system equations, namely the saddle point structure. This symmetry becomes more evident if we discretize system \eqref{eqn:var} in time first, i.e., if we consider 
\begin{subequations}
\label{eqn:var_Rothe}
\begin{eqnarray}
  a(u^n, v) - d(v, p^n) &=& 0, \label{eq:semidiscrete1} \\
  d(\D u^n,q) + c(\D p^n,q) + b(p^n,q) &=& (f^n,q) \label{eq:semidiscrete2}
\end{eqnarray}
\end{subequations}
for all test functions $v\in \V$ and $q\in\Q$. We show that system \eqref{eqn:var_Rothe} is again well-posed. 
\begin{lemma}
\label{lem:timediscrete}
Given $u^{n-1} \in \V$ and $p^{n-1} \in \Q$, system~\eqref{eqn:var_Rothe} is well-posed for all $\tau>0$ and $n = 1, \dots, N$.
\end{lemma}
\begin{proof}
We introduce the bilinear form 
$$
 \bilA([u,p],[v,q]) := a(u,v) - d(v,p) + d(u,q) + c(p,q) + \tau\,b(p,q).
$$
Note that $\bilA$ is coercive, since $\bilA([u,p],[u,p]) = \|u\|^2_a + \|p\|_c^2 + \tau\,\|p\|_b^2$. Furthermore, system \eqref{eqn:var_Rothe} is equivalent to 
\[
  \bilA([u^n,p^n],[v,q]) = \tau (f^n, q) + d(u^{n-1}, q) + c(p^{n-1}, q). 
\]
Thus, the existence of a unique solution follows from the \textit{Lax-Milgram theorem}.
\end{proof}
The property that the $d$-terms in system \eqref{eqn:var_Rothe} cancel for corresponding test functions and summation was not used in the previous approach. We will exploit this property, which then leads to a pair of decoupled multiscale correctors. 
%
\subsubsection{Projections}
In this part, we define the two projections $\Ru$ and $\Rp$, on which the new multiscale method is based. The idea is to use the same projections as for the definition of the spaces $\Vms$ and $\Qms$ in Section~\ref{sss:schweden:spaces}. 
In contrast to the previous approach, they are independent of the bilinear form $d$ such that we do not need an additional fine scale correction. A significant consequence is that the projections are independent of the parameter $\alpha$ despite possible oscillations. 
Furthermore, the two projections are uncoupled, which provides a significant simplification in practical computations.

We define $\Ru\colon \Vh\to\Vms$ and $\Rp\colon \Qh\to\Qms$ by 
\begin{align*}
  a(u - \Ru u, v) = 0,  \qquad
  b(p - \Rp p, q) = 0
\end{align*}
for all $v\in \Vms$ and $q\in \Qms$. Note that we have $\Ru = 1-\calCf^1$ and~$\Rp = 1-\calCf^2$ such that $\Vms= \Ru V_H$ and $\Qms= \Rp Q_H$. The bilinear forms $a$ and $b$ can also be written in terms of operators. On the fine scale, we define $\calA\colon \Vh \to\Vh$ and $\calB\colon \Qh \to \Qh$ by 
\[
  (\calA u, v)_{L^2} := a(u, v), \qquad 
  (\calB p, q)_{L^2} := b(p, q)
\]
for all $v\in \Vh$ and $q\in\Qh$. Note that these operators are only well-defined on the discrete spaces $\Vh$. In the following two lemmata we provide bounds for the introduced projections.
\begin{lemma}\label{lem_projectionA}
The projections $\Ru$ and $\Rp$ satisfy the bounds 
\begin{align*}
  \Vert (1-\Ru) v\Vert 
  &\lesssim H\,\Vert (1-\Ru) v\Vert_1
  \lesssim H\,\Vert v\Vert_1, \\ 
  \Vert (1-\Rp) q\Vert 
  &\lesssim H\,\Vert (1-\Rp) q\Vert_1
  \lesssim H\,\Vert q\Vert_1
\end{align*}
for all $v\in \Vh$ and $q\in \Qh$. 
\end{lemma}
\begin{proof}  
The proof is based on the fact that the bilinear forms $a$ and $b$ are elliptic and the Aubin-Nitsche duality argument. We only show the first estimate here.  
For $v\in \Vh$ we consider the variational problem 
\[
  a(z, w) = (v-\Ru v, w)
\]
with test functions $w\in \Vh$. It was shown in \cite{MalP14} that the ellipticity of $a$ implies
\[
  \Vert z - \Ru z \Vert_1 \lesssim H\, \Vert v-\Ru v \Vert.
\] 
With this, we conclude
\[
  \Vert v-\Ru v\Vert^2
  = a(z, v-\Ru v)
  \lesssim \Vert z - \Ru z\Vert_1 \Vert v-\Ru v\Vert_1
  \lesssim H\, \Vert v-\Ru v \Vert \Vert v-\Ru v\Vert_1.
\]
The final step follows from the stability of the projection. 
\end{proof}
\begin{lemma}\label{lem_projectionB}
The projections $\Ru$ and $\Rp$ are bounded in terms of $\calA$ and $\calB$ by  
\[
  \Vert (1-\Ru) v\Vert_{1} \lesssim H\,\Vert \calA v\Vert, \qquad 
  \Vert (1-\Rp) q\Vert_{1} \lesssim H\,\Vert \calB q\Vert. 
\]
for all $v\in \Vh$ and $q\in \Qh$. 
\end{lemma}
\begin{proof}  
For $v\in \Vh$ we get 
\[
  \Vert v-\Ru v\Vert_1^2
  \lesssim a(v, v-\Ru v)
  = (\calA v, v-\Ru v)
  \le \Vert \calA v\Vert\, \Vert v-\Ru v \Vert.
\]
The claim then follows directly from the previous Lemma~\ref{lem_projectionA}. The proof of the result involving $\calB$ follows the same lines. 
\end{proof}
With the projections $\Ru$ and $\Rp$, we are now able to formulate the new multiscale method. For this we discretize system \eqref{eqn:var_Rothe} in space and consider the problem: 
for each $n = 1, \dots, N$ find $\ums \in \Vms = \Ru \VH$ and $\pms \in \Qms = \Rp \QH$ such that
\begin{subequations}
\label{eqn:ms}
\begin{eqnarray}
    a( \ums, v) - d(v, \pms) &=& 0 , \label{eq:ms1} \\
    d(\D \ums,q) + c(\D \pms, q) + b(\pms,q) &=& (f^n,q)  \label{eq:ms2}
\end{eqnarray}
\end{subequations}
for all $v\in \Vms$ and $q\in \Qms$. Note that this system is again well-posed, cf.~Lemma~\ref{lem:timediscrete}. 
Given $p_\text{ms}^0$, we define the initial value $u_\text{ms}^0$ as before through
\[
  a(u_\text{ms}^0, v) = d(v, p_\text{ms}^0)
\] 
for all $v\in \Vms$.
%
\subsubsection{Convergence}
The aim of this subsection is to prove that the solution provided by \eqref{eqn:ms} approximates the fine scale solution $(\uh, \ph)$ up to order $H$. 
In combination with Theorem~\ref{thm:finescale} this then shows that the multiscale solution converges to the exact solution. 
More precisely, we obtain (assuming $h$ sufficiently small) an error estimate which states that the error is bounded by a constant (independent of $\epsilon$) times $H+\tau$. 
Note that we assume here that the corrector problems are solved exactly. Comments on the localization in practical implementations are given in Section~\ref{sss:local}. 
The main result of this paper reads as follows. 
\begin{theorem}
\label{thm:poro}
Assume $f \in L^\infty(0, T;L^2(D)) \cap H^1(0, T;H^{-1}(D))$ and consistent initial data $u_h^0\in \Vh$, $p_h^0\in \Qh$ as well as $u_\text{ms}^0\in \Vms$ and $p_\text{ms}^0 := \Rp p_h^0\in \Qms$. Then, the error of the multiscale solution compared to the fine scale solution satisfies
\[
  \Vert \uh - \ums\Vert_1 + \Vert \ph - \pms\Vert_1
  \lesssim H \datan + t_n^{-1/2} H\, \Vert p_h^0\Vert_1, 
\]
where $\datan$ is defined by 
\[
 \datan := \Vert p_h^0 \Vert_1 + \Vert f\Vert_{L^2(0, t_n;L^2(D))} + \Vert f\Vert_{L^\infty(0, t_n;L^2(D))} + \Vert \dot f\Vert_{L^2(0, t_n;H^{-1}(D))}.
\]
\end{theorem}
\begin{proof}
As in the proof of the convergence of the multiscale method in \cite{MalP17}, we split the errors in the displacement and pressure into two parts each, namely 
\begin{align*}
  &\errRu := \uh - \uProj, 
  &\errEu := \uProj - \ums, \\ 
  &\errRp := \ph - \pProj, 
  &\errEp := \pProj - \pms.
\end{align*}
Thus, $\rho^n_*$ contains the error of the projections and $\eta^n_*$ the difference of the projection and the multiscale solution. 

{\it Step 1} (estimates of $\rho^n_*$):  
In a first step we bound the projection error due to $\Ru$. For this, we apply Lemma~\ref{lem_projectionB} and use \eqref{eqn:var_d_1}, 
\begin{align*}
  \Vert \errRu \Vert_1
  = \Vert (1-\Ru) \uh \Vert_1 
  \lesssim H\, \Vert \calA \uh \Vert 
  &= H\, \sup_{v_h \in \Vh} \frac{|a(\uh, v_h)|}{\Vert v_h\Vert}  \\
  &= H\, \sup_{v_h \in \Vh} \frac{|d(v_h, \ph)|}{\Vert v_h\Vert}  
  \lesssim H\, \Vert \ph\Vert_1. 
\end{align*}
Note that we have used integration by parts in the last line. Theorem~\ref{thm:stability} then implies that $\Vert \errRu \Vert_1$ is bounded by 
\[
  \Vert \errRu \Vert_1
  \lesssim H\, \big( \Vert p_h^0 \Vert_1 + \Vert f\Vert_{L^2(0, t_n;L^2(D))} \big)
  \le H\,\datan.
\]
Similarly, the projection error due to $\Rp$ can be bounded using \eqref{eqn:var_d_2} by 
\begin{align*}
  \Vert \errRp \Vert_1
  = \Vert (1-\Rp) \ph \Vert_1 
  &\lesssim H\, \Vert \calB \ph \Vert \\
  &= H\, \sup_{q_h \in \Qh} \frac{|b(\ph, q_h)|}{\Vert q_h\Vert}  \\
  &= H\, \sup_{q_h \in \Qh} \frac{1}{\Vert q_h\Vert} |(f^n, q_h) - d(\D\uh, q_h) - c(\D\ph, q_h)|  \\
  &\lesssim H\, \big( \Vert f^n\Vert + \Vert \D\uh\Vert_1 + \Vert \D\ph \Vert \big).
\end{align*}
Using Theorem~\ref{thm:stability}, we obtain the bounds $\Vert \errRp \Vert_1 \lesssim H\, \datan$ if $p_h^0=0$ and $\Vert \errRp \Vert_1 \lesssim t_n^{-1/2} H\, \Vert p_h^0\Vert_1$ in the case $f=0$.

{\it Step 2}:   
In order to bound the remaining errors, we consider specific test functions within the systems \eqref{eqn:var_h} and \eqref{eqn:ms}. Using the definition of $\Ru$, we have for all $v\in\Vms \subseteq \Vh$ that 
\begin{align}
\label{eqn_proof_1}
  a(\errEu, v) - d(v, \errEp)
  = a(\Ru \uh, v) - d(v, \Rp \ph)    
  = a(\uh, v) - d(v, \Rp \ph)   
  = d(v, \errRp).    
\end{align}
Similarly, we have, using the definiton of $\Rp$, for all $q\in\Qms$ that 
\begin{align}
\label{eqn_proof_2}
  d(\D\errEu, q) + c(\D\errEp, q) + b(\errEp, q)
  &= d(\D\Ru\uh, q) + c(\D\Rp\ph, q) + b(\ph, q) - (f^n, q) \notag \\ 
  &= - d(\D\errRu, q) - c(\D\errRp, q).  
\end{align}
Combining equation \eqref{eqn_proof_1} at times $n$ and $(n-1)$, we obtain
\begin{align}
\label{eqn_proof_1b}
  a(\D\errEu, v) - d(v, \D\errEp) = d(v, \D\errRp)
\end{align}
for any $v\in \Vms$. Note that these equations are also valid for $n=1$ beause of the assumed construction of $u_h^0$ and $u^0_\text{ms}$. In order to obtain bounds for $\eta^n_*$, we consider the two cases where either $p_h^0=0$ or $f=0$. An application of the triangle inequality then gives the stated result. 

{\it Step 3} (estimates of $\eta^n_*$ if $p^0_h =0$): 
Note that $p^0_h =0$ also implies $u^0_h = 0$. We now insert the test function $v=\D\errEu$ in \eqref{eqn_proof_1b} and add this to equation \eqref{eqn_proof_2} with $q=\D\errEp$. Together, this yields 
\begin{align*}
  a(\D\errEu, \D\errEu) + c(\D\errEp, &\D\errEp) + b(\errEp, \D\errEp) \\
  &= d(\D\errEu, \D\errRp) - d(\D\errRu, \D\errEp) - c(\D\errRp, \D\errEp)
\end{align*}
and thus, 
\begin{align*}
  \Vert \D\errEu\Vert^2_a + \Vert\D\errEp&\Vert^2_c + b(\errEp, \D\errEp) \\
  &\le C_\alpha \Vert \D\errEu\Vert_1 \Vert \D\errRp\Vert + C_\alpha \Vert \D\errRu\Vert_1 \Vert\D\errEp\Vert + C_M \Vert \D\errRp\Vert\, \Vert\D\errEp\Vert \\
  &\le \frac12 \Vert \D\errEu\Vert_a^2 + \frac12 \Vert\D\errEp\Vert_c^2 + C \Vert \D\errRp\Vert^2 + C' \Vert \D\errRu\Vert_1^2.
\end{align*}
We can eliminate $\Vert \D\errEu\Vert_a$ and $\Vert\D\errEp\Vert_c$ on the right-hand side and multiply the estimate by $2\tau$. Then, the summation over $n$ yields 
\begin{align*}
  \tau \sumjn \Vert \D\errEuj\Vert^2_1 + \tau \sumjn \Vert\D\errEpj\Vert^2 + \Vert \errEp\Vert^2_1
  &\lesssim 2\tau \sumjn \Vert \D\errRpj\Vert^2 + 2\tau \sumjn \Vert \D\errRuj\Vert_1^2. 
\end{align*}
Note that we have used $\errEpinit = 0$. The sum including $\D\errRuj$ can be bounded using once more  Lemma~\ref{lem_projectionB}, 
\begin{align}
  \Vert \D\errRuj\Vert_1
  = \Vert (1-\Ru) \D\uhj \Vert_1 
  &\lesssim H\, \sup_{v_h \in \Vh} \frac{|a(\D\uhj, v_h)|}{\Vert v_h\Vert} \notag  \\
  &= H\, \sup_{v_h \in \Vh} \frac{\big|  d(v_h, \D \phj) \big|}{\Vert v_h\Vert} 
  \lesssim H\,  \Vert \D \phj \Vert_1. \label{eqn_proof_3}
\end{align}
Together with Theorem~\ref{thm:stability} this then leads to 
\begin{align*}
  \tau \sumjn \Vert \D\errRuj\Vert_1^2
  \lesssim \tau\, H^2\, \sumjn \Vert \D \phj \Vert_1^2 
  \lesssim (H\, \datan)^2.
\end{align*}
On the other hand, the sum including $\D\errRpj$ can be bounded with the help of the estimate 
\begin{align*}
  \Vert \D\errRpj\Vert
  = \Vert (1 - \Rp) \D \phj\Vert 
  \lesssim H\, \Vert  \D \phj\Vert_1, 
\end{align*}
which follows from Lemma~\ref{lem_projectionA} and results in 
\begin{align*}
   \tau \sumjn \Vert \D\errRpj\Vert^2
   \le \tau \sumjn H^2\, \Vert  \D \phj\Vert^2_1
  \lesssim (H\, \datan)^2.
\end{align*}
This does not only provide the bound $\Vert \errEp\Vert_1 \lesssim H\, \datan$ but also, using \eqref{eqn_proof_1},
\[
  \Vert \errEu\Vert_1 
  \lesssim \Vert \errRp\Vert + \Vert \errEp\Vert
  \lesssim H\, \datan.
\]
{\it Step 4} (estimates of $\eta^n_*$ if $f=0$):   
We emphasize that also in this case we have $\eta^0_p=0$ by assumption. Together with \eqref{eqn_proof_1} this yields for $\eta_u^0$ the estimate 
\[
  \Vert \eta^0_u\Vert^2_1
  \lesssim a(\eta^0_u, \eta^0_u)
  = d(\eta^0_u, \eta^0_p) + d(\eta^0_u, \rho^0_p)
  \lesssim \Vert\eta^0_u\Vert_1 \Vert \rho^0_p \Vert
\]
and thus $\Vert \eta^0_u\Vert_1 \lesssim \Vert \rho^0_p \Vert \lesssim H\, \Vert p^0_h \Vert_1$. 
Note that it is sufficient to bound $\Vert \errEp\Vert_1$ in terms of $H\, \datan$, since we have $\Vert \errEu\Vert_1 \lesssim \Vert \errRp\Vert + \Vert \errEp\Vert$ by \eqref{eqn_proof_1}. 
As in step 3 we consider the sum of equation~\eqref{eqn_proof_1b} with $v = \D \errEu$ and equation~\eqref{eqn_proof_2} with $q = \D \errEp$. 
Multiplying the result by $2\tau$, we get
\begin{align*}
  2\tau \Vert \D\errEu\Vert^2_a + 2 \tau \Vert\D\errEp\Vert^2_c + \Vert \errEp\Vert_b^2 - \Vert \errEppre\Vert_b^2 
  \lesssim 2\tau  \Vert \D\errRp\Vert^2 +  2\tau \Vert \D\errRu\Vert_1^2.
\end{align*}
Another multiplication by $t_n^2$ and the estimate $t_n^2-t^2_{n-1} \le 3\tau t_{n-1}$ then lead to 
\begin{align*}
  2\tau t^2_n \Vert \D\errEu\Vert^2_a + 2 \tau t^2_n \Vert\D\errEp&\Vert^2_c + t^2_n \Vert \errEp\Vert_b^2 - t^2_{n-1} \Vert \errEppre\Vert_b^2  \\
  &\lesssim 2\tau t^2_n \Vert \D\errRp\Vert^2 +  2\tau t^2_n \Vert \D\errRu\Vert_1^2
  + 3\tau t_{n-1} \Vert \errEppre\Vert_1^2.
\end{align*}
Taking the sum, we obtain 
\begin{align}
  \tau \sum^n_{j=1} t_j^2 \Vert \D\errEuj\Vert_1^2 + t^2_n \Vert \errEp\Vert_1^2
  &\lesssim \tau \sum^n_{j=1} t_j^2 \Vert \D\errEuj\Vert_a^2 + \sum_{j=1}^n  \Big( t^2_j \Vert \errEpj\Vert_b^2 - t^2_{j-1} \Vert \errEpjj\Vert_b^2 \Big) \notag \\
  &\lesssim \tau \sum_{j=1}^n t^2_j \Vert \D\errRpj\Vert^2 + \tau \sum_{j=1}^n  t^2_j \Vert \D\errRuj\Vert_1^2 + \tau\sum_{j=1}^{n-1}  t_j \Vert \errEpj\Vert_1^2. \label{eqn_proof_4}
\end{align}
To bound the first sum, we apply first Lemma~\ref{lem_projectionA} and then Theorem~\ref{thm:stability},
\[
  \tau \sum_{j=1}^n t^2_j \Vert \D\errRpj\Vert^2
  \lesssim \tau\sum_{j=1}^n t^2_j H^2 \Vert \D \phj\Vert_1^2
  \lesssim \tau\sum_{j=1}^n  H^2 \Vert p_h^0 \Vert^2
  = t_n H^2 \Vert p_h^0 \Vert^2_1. 
\]
For the second sum we use the estimate $\Vert \D\errRuj\Vert_1 \lesssim H\,  \Vert \D \phj \Vert_1$ from \eqref{eqn_proof_3}, which is also valid for non-zero initial values. With Theorem~\ref{thm:stability} this then leads to 
\[
  \tau\sum_{j=1}^n t^2_j \Vert \D\errRuj\Vert_1^2
  \lesssim \tau\sum_{j=1}^n t^2_j H^2  \Vert \D \phj \Vert_1^2
  \lesssim t_n H^2 \Vert p_h^0 \Vert^2_1.
\]
{\it Step 5} (estimate of the last sum in \eqref{eqn_proof_4}): 
In order to bound the third sum on the right-hand side of \eqref{eqn_proof_4}, we consider once more the sum of equations~\eqref{eqn_proof_1} and \eqref{eqn_proof_2}. For test functions $v = \D \errEu$ and $q = \errEp$ we get after multiplication with $2\tau t_n$ and application of Youngs inequality
\begin{align*}
  t_n \big( \Vert\errEu\Vert^2_a - \Vert \errEupre&\Vert^2_a  \big) 
  + t_n \big( \Vert\errEp\Vert^2_c - \Vert \errEppre\Vert^2_c  \big)
  + 2\tau t_n \Vert \errEp\Vert_b^2 \\
%
%
  &\lesssim  \lambda\tau t_n^2 \Vert \D\errEu\Vert^2_1 + \lambda^{-1}\tau \Vert\errRp\Vert^2  
  +  \tau t_n^2 \Vert \D\errRu\Vert^2 
  +  \tau t_n^2 \Vert \D\errRp\Vert^2 + \tau \Vert\errEp\Vert^2
\end{align*}
for any $\lambda>0$. 
We add $\tau \Vert \errEupre\Vert^2_a + \tau \Vert \errEppre\Vert^2_c$ on both sides and take the sum over $n$ such that we obtain 
\begin{align*}
  t_n \Vert\errEu\Vert^2_1 + t_n \Vert\errEp\Vert^2 
  + \sum_{j=1}^n \tau t_j &\Vert \errEpj\Vert_1^2 
  \lesssim \lambda \underbrace{\tau \sum^n_{j=1} t_j^2 \Vert \D\errEuj\Vert^2_1}_{\textcircled{\scriptsize \texttt{a}}} 
  + \frac 1 \lambda \underbrace{\tau \sum^n_{j=1} \Vert\errRpj\Vert^2 }_{\textcircled{\scriptsize \texttt{b}}}  \\
  &+  \underbrace{\tau\sum_{j=1}^n t_j^2\, \big( \Vert \D\errRuj\Vert^2 + \Vert \D\errRpj\Vert^2 \big) }_{\textcircled{\scriptsize \texttt{c}}}
  + \underbrace{ \tau\sum^n_{j=1} \big( \Vert\errEpj\Vert^2
  + \Vert \errEujj\Vert^2_1 \big) }_{\textcircled{\scriptsize \texttt{d}}}. 
\end{align*}
Note that the sum on the left-hand side is the term we aim to bound. For a sufficiently small $\lambda$ (depending only on the generic constant of the estimates) we can eliminate ${\textcircled{\scriptsize \texttt{a}}}$ with the left-hand side in \eqref{eqn_proof_4}. For the remaining three parts on the right-hand side we estimate
\[
  \textcircled{\scriptsize \texttt{b}}
  = \tau\sum^n_{j=1} \Vert\errRpj\Vert^2
  \lesssim \tau\sum^n_{j=1} H^2 \Vert \phj\Vert^2_1
  \lesssim \tau\sum^n_{j=1} H^2 \Vert p^0_h\Vert^2_1
  = t_n H^2 \Vert p_h^0\Vert^2_1
\]
and, with Lemma~\ref{lem_projectionA} and Theorem~\ref{thm:stability}, 
\[
  \textcircled{\scriptsize \texttt{c}}
  \lesssim \tau\sum_{j=1}^n H^2\,t_j^2\, \Big( \Vert \D \uhj\Vert_1^2 + \Vert \D \phj\Vert_1^2 \Big)
  \lesssim \tau\, (t_n+1) \sum_{j=1}^n H^2\, \Vert p_h^0 \Vert_1^2 
  = (t_n^2+t_n)\, H^2\, \Vert p_h^0 \Vert_1^2 .
\]
Finally, with the equations~\eqref{eqn_proof_1} and \eqref{eqn_proof_2} and test functions $v=\errEu$ and $q=\errEp$ one can show as in \cite{MalP17} that also $\textcircled{\scriptsize \texttt{d}} \lesssim t_n H^2 \Vert p_h^0\Vert^2_1$. 
In summary, this yields 
\[
  \Vert \errEp\Vert_1
  \lesssim (1+t_n^{-1/2})\, H\, \Vert p_h^0\Vert_1.
  \qedhere
\]
\end{proof}
This theorem shows together with Theorem~\ref{thm:finescale} that the multiscale method proposed in \eqref{eqn:ms} converges linearly as $H+\tau$. For this we consider the $L^\infty(0,T;V)$ norm for $u$ and the $L^\infty(0,T;L^2(D)) \cap L^2(0,T;Q)$ norm for $p$. We emphasize that the involved constants are independent of derivatives of the coefficients $\mu$, $\lambda$, $\kappa$, and $\alpha$. 
%
\subsubsection{Localization}\label{sss:local}
The convergence result in Theorem~\ref{thm:poro} assumes that the basis functions of the space $\Vms$ and $\Qms$ are given. In practical simulations, however, they need to be approximated themselves. These basis functions have global support in the domain $D$ but, as shown in \cite{MalP14}, decay exponentially, see also \cite{KY16,KPY18} for an alternative constructive proof. This property allows to consider a truncation of the basis functions. More precisely, the basis functions in $\Vms$ and $\Qms$ are computed by solving local problems. Given the basis function $\lambda_z\in V_H$ to an inner node $z$, the corresponding multiscale basis function~$\Ru \lambda_z$ is computed as in Section~\ref{sss:schweden:spaces} but with the computational domain restricted to a subdomain of $\ell$ additional coarse element layers surrounding the support of~$\lambda_z$. The so-called {\em localization parameter} $\ell$ defines a new discretization parameter. More details on the practical computation can be found in \cite{HenP13}. 

Note that the convergence result in Theorem~\ref{thm:poro} remains valid if the localization parameter $\ell$ is chosen sufficiently large, i.e., $\ell \approx \log H$.
%
\section{Numerical Examples}\label{sec:numex}
In order to assess the method numerically, we consider numerical examples in two and three space dimensions. We measure the error in the discrete time-dependent norm
\begin{equation*}
 \|(v,q)\|^2_{D,N} := \sum_{i=1}^N \tau \Big(\|\nabla v^i\|^2 + \|\nabla q^i\|^2\Big)
\end{equation*}
with $N=T/\tau$ the number of time steps and $v = \{v^i\}_{i=1}^N$, $q = \{q^i\}_{i=1}^N$. The corresponding relative error between the multiscale solution and the fine scale solution is then defined by
\begin{equation*}
 \|(u_\text{ms},p_\text{ms})-(u_h,p_h)\|_\mathrm{rel} = \frac{\|(u_\text{ms},p_\text{ms})-(u_h,p_h)\|_{D,N}}{\|(u_h,p_h)\|_{D,N}}.
\end{equation*}
Further, we set $D := (0,1)^d$ as the domain and $T := 1$ as final time with time step size $\tau = 0.01$ (and thus $N = 100$) for the examples with $d=2$ and $\tau = 0.05$ (and thus $N = 20$) for the example with $d=3$.

The reference solution $(u_h,p_h)$ is computed on a regular uniform mesh $\mathcal{T}_h$ consisting of elements with given mesh size $h$.
The local corrector problems are also solved on patches with mesh size $h$. The parameters are chosen to be piecewise constant on elements of $\mathcal{T}_\epsilon$ and the value is obtained as a uniformly distributed random number between two given bounds, i.e., for any $K \in \mathcal{T}_\epsilon$ we have
\begin{equation}
\label{eq:coeff}
  \kappa|_{K} \sim U[0.1,0.12], \ \
  \mu|_{K} \sim U[32.2,62.2], \ \
  \lambda|_{K} \sim U[40.98,60.98],\ \
  \alpha|_{K} \sim U[0.5,1]
\end{equation}
and $M = \nu = 1$, where $\mathcal{T}_\epsilon$ is a mesh with mesh size $\epsilon > h$ to guarantee that the reference solution is reasonable. Note that we take representative global samples for the above parameters. In all numerical tests, the localization parameter from Section~\ref{sss:local} is set to $\ell = 2$ which showed to be sufficient. It should be mentioned that the choice of the localization parameter generally needs to be increased for smaller values of $H$ and may be decreased for larger $H$, see \cite{HenP13} for details. 
The computations are done using an adaption of the code from \cite{Hel17}. For a detailed description about the implementation of the LOD method, we further refer to \cite{EngHMP16}.
%
%
\subsection{Two-dimensional examples}
In all two-dimensional experiments, the fine mesh size is set to $h = \sqrt{2}\cdot2^{-8}$ and $\epsilon = \sqrt{2}\cdot2^{-6}$.

For the first example we set $f = 1$ and $p^0(x) = (1-x_1)\,x_1\,(1-x_2)\,x_2$.  We prescribe homogeneous Dirichlet boundary conditions for $p$ on $\partial D$, homogeneous Dirichlet boundary conditions for $u$ on $\{x\in\partial D\colon x_2 = 0 \text{ or } x_2 = 1\}$ and homogeneous 
Neumann boundary conditions on $\{x\in\partial D\colon x_1 = 0 \text{ or } x_1 = 1\}$. The results for different values of $H$ are shown
in Figure~\ref{fig_exp12} (left, {\protect \tikz{ \draw[line width=1.5pt, myRed] (0,0) -- (0.13,0.23) -- (0.26, 0) -- cycle;}} ). The plot indicates a convergence rate even slightly better than $1$ with respect to the coarse mesh size $H$ and becomes steeper for smaller values of~$H$, since the $\epsilon$-scale is almost resolved. 

In the second example we consider $p^0(x) = \sqrt{1-x_2}$, $f = 0$, and $u$ and $p$ fulfill homogeneous Dirichlet boundary conditions on $\{x\in\partial D \colon x_2 = 1\}$ and homogeneous Neumann boundary conditions on the remaining parts of  $\partial D$. Here, the same behavior as in the last example may be observed, cf.~Figure~\ref{fig_exp12} (left, {\protect \tikz{ \draw[line width=1.5pt, myBlue] circle (0.6ex);}}). 

Figure~\ref{fig_exp12} (left, {\protect \tikz{ \draw[line width=1.5pt, myGreen] (0,0) rectangle (0.23,0.23);}}) shows the results of the third example, where $f$ is chosen as random fine scale finite element function with values between $0$ and $1$, $p^0(x) = (1-x_2)\,x_2$, and with the same boundary conditions as in the previous experiment. The plot shows the predicted linear convergence. In this example, the error curve does not become steeper in the regime $H \leq \epsilon$, which may be related to the fact that $f$ is a function on the fine scale. 
\begin{figure}
\begin{center}
\scalebox{0.8}{
%
%
\begin{tikzpicture}

\begin{axis}[%
width=2.8in,
height=2.2in,
at={(0.952in,0.761in)},
scale only axis,
xmode=log,
xmin=0.008,
xmax=1,
xminorticks=true,
xlabel={mesh size $H$},
ymode=log,
ymin=0.0005,
ymax=1,
yminorticks=true,
axis background/.style={fill=white},
legend style={legend cell align=left, align=left, draw=white!15!black},
legend pos=north west
]
\addplot [color=myRed, line width=1.5pt, mark size=4.0pt, mark=triangle]
  table[row sep=crcr]{%
0.707106781186548	0.734892773683\\
0.353553390593274	0.349978849365\\
0.176776695296637	0.147004979476\\
0.0883883476483184	0.0562541421158\\
0.0441941738241592	0.0202724915313\\
0.0220970869120796	0.00683569790748\\
0.0110485434560398	0.00166090502517\\
};
\addlegendentry{exp 1}

\addplot [color=myBlue, line width=1.5pt, mark size=3.5pt, mark=o]
  table[row sep=crcr]{%
0.707106781186548	0.184464380957\\
0.353553390593274	0.0915540214912\\
0.176776695296637	0.0414402260784\\
0.0883883476483184	0.0166536896593\\
0.0441941738241592	0.00598141247855\\
0.0220970869120796	0.00235673892368\\
0.0110485434560398	0.000688962695079\\
};
\addlegendentry{exp 2}

\addplot [color=myGreen, line width=1.5pt, mark size=3.5pt, mark=square]
  table[row sep=crcr]{%
0.707106781186548	0.381265896956\\
0.353553390593274	0.165022948516\\
0.176776695296637	0.0670299522273\\
0.0883883476483184	0.0263145171409\\
0.0441941738241592	0.0118474161589\\
0.0220970869120796  0.00695731622328\\
0.0110485434560398	0.0037295619875\\
};
\addlegendentry{exp 3}

\addplot [color=black, line width=0.7pt, dashed]
  table[row sep=crcr]{%
0.7		0.07\\
0.0115	0.00115\\
};
\addlegendentry{order 1}

\end{axis}
\end{tikzpicture}%
}
	\scalebox{0.8}{
%
%
\begin{tikzpicture}

\begin{axis}[%
width=2.8in,
height=2.2in,
at={(0.949in,0.696in)},
scale only axis,
xmode=log,
xmin=0.11,
xmax=0.8,
xminorticks=true,
xlabel style={font=\color{white!15!black}},
xlabel={mesh size $H$},
ymode=log,
ymin=0.01,
ymax=1,
yminorticks=true,
axis background/.style={fill=white},
legend style={legend cell align=left, align=left, draw=white!15!black},
legend pos=north west
]
\addplot [color=myOrange, line width=1.5pt, mark size=4.0pt, mark=*]
  table[row sep=crcr]{%
0.707106781186548	0.454243247838\\
0.353553390593274	0.217049922349\\
0.235702260395516	0.124374103655\\
0.117851130197758	0.0347300345963\\
};
\addlegendentry{3D example}

\addplot [color=black, line width=0.7pt, dashed]
  table[row sep=crcr]{%
0.72	0.12\\
0.12	0.02\\
};
\addlegendentry{order 1}

\end{axis}
\node[rotate = 90] at (0.5, 4.5) {error in $\|\cdot\|_\mathrm{rel}$};
\end{tikzpicture}%
}
\end{center}
\caption{Errors of the multiscale method in two (left) and three space dimensions (right).}
\label{fig_exp12}
\end{figure}
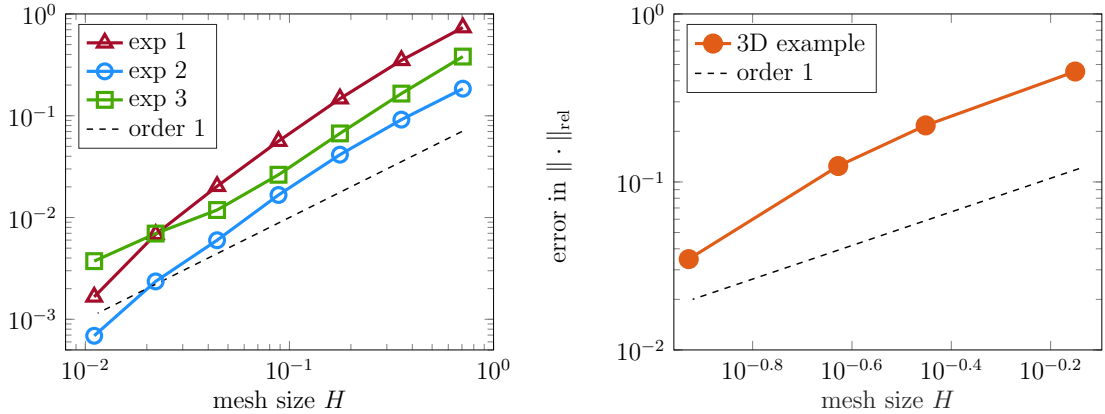
%
\subsection{Three-dimensional example}
For the three-dimensional setting, we restrict ourselves to $h = \sqrt{2}/3\cdot2^{-3}$ and $\epsilon = \sqrt{2}/3\cdot2^{-2}$ due to the high computational complexity. We choose the coefficients as in \eqref{eq:coeff}, set $f = 0$, $p^0(x) =  (1-x_1)\,x_1\,(1-x_2)\,x_2\,(1-x_3)\,x_3$, homogeneous Dirichlet boundary conditions on $\{x\in\partial D \colon x_3 = 1\}$ and homogeneous Neumann boundary conditions on the remaining parts of  $\partial D$, and $\ell = 2$ as before. The errors for this example are plotted in Figure~\ref{fig_exp12} (right, {\protect \tikz{ \fill[myOrange] circle (0.7ex);}}) and are mainly to indicate that the three-dimensional setting can be handled if appropriate computing capacities are available.
%
%
\section{Conclusions}
Within this paper, we have considered a poroelastic model problem with rapidly  oscillating material parameters. The proposed multiscale finite element method is based on the LOD method and exploits the saddle point structure of the system. In contrast to the classical approach based on the static equations, which leads to coupled corrector problems, this method enables decoupled corrections for displacement and pressure such that no additional fine scale corrections are necessary. Although the correctors are independent of the coefficient $\alpha$, we are able to prove first-order convergence of the method, which is also illustrated by numerical experiments.

Future research aims to study further enrichment via local eigenvalue computations to control the reliability of the approach in the case when the coefficients model high contrast inclusions and channels, as well as fractured regions \cite{ChuEL18b, ChuEL18, PetS16,doi:10.1137/17M1147305}.

\bibliographystyle{plain}

\begin{thebibliography}{10}
	
	\bibitem{Biot41}
	M.~A. Biot.
	\newblock General theory of three-dimensional consolidation.
	\newblock {\em J. Appl. Phys.}, 12(2):155--164, 1941.
	
	\bibitem{doi:10.1137/17M1147305}
	D.~L. Brown, J.~Gedicke, and D.~Peterseim.
	\newblock Numerical homogenization of heterogeneous fractional {L}aplacians.
	\newblock {\em Multiscale Model. Simul.}, 16(3):1305--1332, 2018.
	
	\bibitem{Brown.Peterseim:2014}
	D.~L. {Brown} and D.~{Peterseim}.
	\newblock A multiscale method for porous microstructures.
	\newblock {\em Multiscale Model. Simul.}, 14:1123--1152, 2016.
	
	\bibitem{BroV16a}
	D.~L. Brown and M.~Vasilyeva.
	\newblock A generalized multiscale finite element method for poroelasticity
	problems {I}: {L}inear problems.
	\newblock {\em J. Comput. Appl. Math.}, 294:372--388, 2016.
	
	\bibitem{BroV16b}
	D.~L. Brown and M.~Vasilyeva.
	\newblock A generalized multiscale finite element method for poroelasticity
	problems {II}: {N}onlinear coupling.
	\newblock {\em J. Comput. Appl. Math.}, 297:132--146, 2016.
	
	\bibitem{doi:10.1137/130921866}
	A.~Caiazzo and J.~Mura.
	\newblock Multiscale modeling of weakly compressible elastic materials in the
	harmonic regime and applications to microscale structure estimation.
	\newblock {\em Multiscale Model. Simul.}, 12(2):514--537, 2014.
	
	\bibitem{ChuEL18b}
	E.~T. Chung, Y.~Efendiev, and W.~T. Leung.
	\newblock Constraint energy minimizing generalized multiscale finite element
	method.
	\newblock {\em Comput. Method. Appl. M.}, 339:298--319, 2018.
	
	\bibitem{ChuEL18}
	E.~T. Chung, Y.~Efendiev, and W.~T. Leung.
	\newblock Fast online generalized multiscale finite element method using
	constraint energy minimization.
	\newblock {\em J. Comput. Phys.}, 355:450--463, 2018.
	
	\bibitem{Cia78}
	P.~G. Ciarlet.
	\newblock {\em The Finite Element Method for Elliptic Problems}.
	\newblock North-Holland, Amsterdam, 1978.
	
	\bibitem{Cia88}
	P.~G Ciarlet.
	\newblock {\em Mathematical elasticity. Vol. {I}}.
	\newblock North-Holland, Amsterdam, 1988.
	
	\bibitem{EGH13}
	Y.~Efendiev, J.~Galvis, and T.~Y. Hou.
	\newblock Generalized multiscale finite element methods ({GM}s{FEM}).
	\newblock {\em J. Comput. Phys.}, 251:116--135, 2013.
	
	\bibitem{EngHMP16}
	C.~{Engwer}, P.~{Henning}, A.~{M{\aa}lqvist}, and D.~{Peterseim}.
	\newblock Efficient implementation of the localized orthogonal decomposition
	method.
	\newblock {\em ArXiv e-prints}, 2016.
	\newblock 1602.01658.
	
	\bibitem{ErnM09}
	A.~Ern and S.~Meunier.
	\newblock A posteriori error analysis of {E}uler-{G}alerkin approximations to
	coupled elliptic-parabolic problems.
	\newblock {\em ESAIM Math. Model. Numer. Anal.}, 43(2):353--375, 2009.
	
	\bibitem{doi:10.1137/16M1088533}
	D.~Gallistl and D.~Peterseim.
	\newblock Computation of quasi-local effective diffusion tensors and
	connections to the mathematical theory of homogenization.
	\newblock {\em Multiscale Model. Simul.}, 15(4):1530--1552, 2017.
	
	\bibitem{Hel17}
	F.~Hellman.
	\newblock Gridlod.
	\newblock \url{https://github.com/fredrikhellman/gridlod}, 2017.
	\newblock GitHub repository, commit 3e9cd20970581a32789aa1e21d7ff3f7e8f0b334.
	
	\bibitem{HenP13}
	P.~Henning and D.~Peterseim.
	\newblock Oversampling for the multiscale finite element method.
	\newblock {\em Multiscale Model. Simul.}, 11(4):1149--1175, 2013.
	
	\bibitem{KPY18}
	R.~Kornhuber, D.~Peterseim, and H.~Yserentant.
	\newblock An analysis of a class of variational multiscale methods based on
	subspace decomposition.
	\newblock {\em Math. Comp.}, 87:2765--2774, 2018.
	
	\bibitem{KY16}
	R.~Kornhuber and H.~Yserentant.
	\newblock Numerical homogenization of elliptic multiscale problems by subspace
	decomposition.
	\newblock {\em Multiscale Model. Simul.}, 14(3):1017--1036, 2016.
	
	\bibitem{MalP17}
	A.~M{\aa}lqvist and A.~Persson.
	\newblock A generalized finite element method for linear thermoelasticity.
	\newblock {\em ESAIM Math. Model. Numer. Anal.}, 51(4):1145--1171, 2017.
	
	\bibitem{MalP18}
	A.~M{\aa}lqvist and A.~Persson.
	\newblock Multiscale techniques for parabolic equations.
	\newblock {\em Numer. Math.}, 138(1):191--217, 2018.
	
	\bibitem{MalP14}
	A.~M{\aa}lqvist and D.~Peterseim.
	\newblock Localization of elliptic multiscale problems.
	\newblock {\em Math. Comp.}, 83(290):2583--2603, 2014.
	
	\bibitem{Mura2016}
	J.~Mura and A.~Caiazzo.
	\newblock {\em A Two-Scale Homogenization Approach for the Estimation of
		Porosity in Elastic Media}, pages 89--105.
	\newblock Springer International Publishing, Cham, 2016.
	
	\bibitem{Pet16}
	D.~Peterseim.
	\newblock Variational multiscale stabilization and the exponential decay of
	fine-scale correctors.
	\newblock In {\em Building Bridges: Connections and Challenges in Modern
		Approaches to Numerical Partial Differential Equations}, pages 341--367.
	Springer, 2016.
	
	\bibitem{PetS16}
	D.~Peterseim and R.~Scheichl.
	\newblock Robust numerical upscaling of elliptic multiscale problems at high
	contrast.
	\newblock {\em Comput. Methods Appl. Math.}, 16(4):579--603, 2016.
	
	\bibitem{Sho00}
	R.~E. Showalter.
	\newblock Diffusion in poro-elastic media.
	\newblock {\em J. Math. Anal. Appl.}, 251(1):310--340, 2000.
	
	\bibitem{Zob10}
	M.~D. Zoback.
	\newblock {\em Reservoir Geomechanics}.
	\newblock Cambridge University Press, 2010.
	
\end{thebibliography}

\end{document}